\documentclass[11pt]{amsart}

\usepackage{amsfonts,amsmath,amssymb}
\usepackage{enumerate}

\numberwithin{equation}{section}

\newtheorem{theorem}{Theorem}[section]
\newtheorem{lemma}[theorem]{Lemma}
\newtheorem{proposition}[theorem]{Proposition}

\newtheorem{corollary}[theorem]{Corollary}
\newtheorem{remark}[theorem]{Remark}

\newcommand{\bb}[1]{\mathbb{#1}}
\newcommand{\sca}[1]{\left(#1\right)} 

\newcommand{\nor}[1]{\left\Vert #1\right\Vert}

\usepackage[pdftex,colorlinks]{hyperref}

\begin{document}

\title[Idempotents of large norm in Fourier algebras]
{Idempotents of large norm and homomorphisms of Fourier algebras}

\author{M. Anoussis,  G. K. Eleftherakis, A. Katavolos}

\address{M. Anoussis\\Department of Mathematics\\ University of the Aegean\\
832 00 Karlovassi, Greece }\email{mano@aegean.gr}

\address{G. K. Eleftherakis\\Department of Mathematics\\ Faculty of Sciences\\
University of Patras\\265 00 Patra, Greece }
\email{gelefth@math.upatras.gr}

\address{A. Katavolos\\Department of Mathematics\\
National and Kapodistrian University of Athens\\
1578 84 Athens, Greece }\email{akatavol@math.uoa.gr}


\maketitle

\begin{abstract}

 We  provide necessary and sufficient conditions 
for the existence of idempotents of arbitrarily large norms in the Fourier algebra 
 $A(G)$ and the Fourier-Stieltjes algebra $B(G)$ of a locally compact group $G$.
We prove that  the existence of idempotents of arbitrarily large norm in $B(G)$
implies the existence of 
homomorphisms  of arbitrarily large norm from $A(H)$ into $B(G)$ for every locally compact group $H$.
A partial converse is also obtained:  the  existence of homomorphisms  of arbitrarily large norm 
from $A(H)$ into $B(G)$ for some amenable locally compact group $H$ implies 
the existence of idempotents of arbitrarily large norm in $B(G)$.
\end{abstract}

\section{introduction}
Let $G$  be a locally compact group. 
The Fourier algebra $A(G)$   and the Fourier-Stieltjes algebra $B(G)$ of
the group $G$ were introduced by Eymard \cite{eym}. 
The Fourier-Stieltjes algebra of $G$ consists of 
the matrix coefficients $\sca{\pi(\cdot)\xi,\eta}$ of all continuous
unitary representations $\pi$ of $G$, while the Fourier algebra of $G$ consists of the matrix 
coefficients of the left regular  representation of $G$. If $G$ is abelian, $A(G)$ and $B(G)$ 
can be identified, via the Fourier transform, with $L^1(\widehat{G})$ and the measure algebra
$M(\widehat{G})$ of the dual group $\widehat{G}$, respectively.
In \cite{co1} Cohen characterized the homomorphisms from $A(H)$ into $B(G)$ in terms of 
piecewise affine maps  when $H, G$ are locally compact abelian groups. 
To obtain his result he proved a characterization of idempotents in $B(G)$ \cite{co2}.
Host in \cite{ho} extended  the characterization  of idempotents in $B(G)$ 
to general locally compact groups.

Homomorphisms of Fourier algebras for locally compact groups were studied by Ilie \cite{I}  
and Ilie and Spronk  \cite{IS}.
They characterized completely bounded homomorphisms   from $A(H)$ into $B(G)$ 
for locally compact groups 
$H, G$ with $H$ amenable in terms of piecewise affine maps \cite{IS} (see also \cite{daws}).

Let $G, H$ be locally compact  groups. Let   $K$ be  a subgroup of $G$  and $C$ 
a left coset of $K$  in $G$. A map $\alpha: C \rightarrow H$ is
called affine  if there exists  a continuous homomorphism  $\theta : K\rightarrow H$ 
and elements $s_0\in H, t_0\in G$ such that $C=t_0^{-1}K$ and 
$$\alpha(t)=s_0\theta (t_0t), $$ 
for all $t \in C$.

A map $\alpha: Y\rightarrow H$ is called {\it  piecewise affine} 
if $Y$ can be written as a disjoint union $Y=\bigcup\limits_{i=1}^mY_i$, where each $Y_i$ 
belongs to  the open coset ring $\Omega_0(G)$,   such that each restriction $\alpha|_{Y_i}$ 
extends to an affine  map 
$\alpha_i: C_i\rightarrow H$  defined on an open coset $C_i\supseteq Y_i$. 
 
Recall that the open coset ring $\Omega_0(G)$
is the ring generated by the open cosets of the group $G$.

Let  $Y$ be an open and closed  subset of $G$ and  $\alpha$ 
a piecewise affine map  $Y \rightarrow H$. Define $\rho: A(H)\rightarrow B(G)$ by 
\begin{equation} \label{star2} 
\rho (u)(t)=\left\{\begin{array}{ll} u\!\circ\!\alpha (t), & t\in Y\\  
0, & t\in G\setminus Y \end{array}\right.
\end{equation}
 Ilie and Spronk proved in 
\cite[Proposition 3.1 and Theorem 3.7]{ IS} that $\rho$
is a completely bounded homomorphism and that,  if the group $H$ is amenable, 
every completely bounded homomorphism $\rho: A(H)\rightarrow B(G)$ is of this form.

\smallskip
{\it Notation}
The symbol $\chi_F$ denotes the characteristic function of a set $F$.
\smallskip

 To motivate our work, consider the following simple example 
of completely bounded homomorphisms 
from $A(\mathbb Z)$  to $A(\mathbb Z)$    of arbitrarily large norm:

For  $F=\{-k, \dots, k\}\subseteq \bb Z$  we define the map
$\rho_F: A(\bb Z)\to A(\bb Z)$ by 
$$\rho_F (u)(j)= \begin{cases} u(0) & \text{if } j\in F\\
								0  & \text{if } j\notin F
\end{cases}$$
Then it follows from  \cite{I} that $\rho_F$
is  a completely bounded homomorphism.
Consider the function $u_0 :\bb Z \to \bb Z$ given by $u_0(i)=\delta_{i,0}$.
Clearly $u_0\in A(\bb Z)$ and   $\|u_0\|_{A(\bb Z)}\leq 1$.

Since $ \rho_F (u_0)=\chi_F$, its  Fourier transform is 
$ \widehat{\rho_F (u_0)}(z)= \widehat{\chi_F} (z)=\sum_{i\in F}z^{-i} $ 
and so
$$
\|\rho_F\| \geq \|\rho_F(u_0)\|_{A(\bb Z)}= \|\chi_F\|_{A(\bb Z)}=\|\widehat{\chi_F}
 \|_{L^1(\bb T)}$$ $$= \int_{\bb T} \left| \sum_{i\in F}z^{-i}\right|dz=\int_0^{2\pi} |D_k(x)|\frac{dx}{2\pi}
$$
where $D_k$ is the Dirichlet kernel, 
and it is known that the $L^1$ norm of $D_k$ grows like $\log k$.

In the above example  we used the existence of idempotents in $A(\bb Z)$ of large norm 
to construct homomorphisms  of $A(\bb Z)$  with large norm.

In this work we study the following questions: a) for which locally compact groups $G$ there exist  
idempotents of arbitrarily large norms in the Fourier algebra 
 $A(G)$ (resp. in the Fourier-Stieltjes algebra $B(G)$) and b) how is the existence of 
 idempotents of arbitrarily large norm related to the existence of 
homomorphisms  of arbitrarily large norm between  Fourier algebras.
 
We  provide necessary and sufficient conditions 
for the existence of idempotents of arbitrarily large norms in the Fourier algebra 
$A(G)$ (resp. in the Fourier-Stieltjes algebra $B(G)$) of a locally compact group $G$.
To prove our results we reduce the problem to the case where $G$ is totally disconnected. Then 
we first consider the case where $G$ is a discrete group in Proposition \ref{m222}, 
and for the general case we use a result of  Leiderman,  Morris and  Tkachenko 
for  totally disconnected groups  \cite{lm}. 
We also prove that  the existence of idempotents of arbitrarily large norm in $B(G)$
implies the existence of homomorphisms  of arbitrarily large norm from $A(H)$ 
into $B(G)$ for every locally compact group $H$.
Finally, we obtain a partial converse:  the  existence of 
homomorphisms  of arbitrarily large norm from $A(H)$ into $B(G)$ 
for some amenable locally compact group $H$ implies 
the existence of idempotents of arbitrarily large norm in $B(G)$.

\section{norms of idempotents}

Let $G$ be a locally compact group. 
A function  $u: G\rightarrow\bb C$ 
is called a multiplier of $A(G)$ if $uA(G)\subseteq A(G).$ In this case the map 
$m_u: A(G)\to A(G): v\mapsto uv$ is bounded. In case it is completely bounded we call $u$ 
a completely bounded multiplier. We denote by 
$M_{cb}A(G)$ the algebra of completely bounded multipliers of $A(G)$. 

The space
$M_{cb}A(G)$ inherits the operator space structure from the space 
$CB(A(G))$  of completely bounded maps 
$A(G)\rightarrow A(G)$.  We write $ \|u\|_{CB(A(G))}$, and simply   $\|u\|_{cb}$ 
when there is no risk of confusion,  for the completely bounded norm $ \|m_u\|_{CB(A(G))}$
of an element $u\in M_{cb}A(G)$. Note that $B(G)$ consists of completely bounded multipliers
on $A(G)$ \cite[Corollary 1.8]{decahaa}; 
thus $B(G)$ (and also $A(G)$) inherits the operator space structure from $M_{cb}A(G)$. 

It is shown in \cite{decahaa} that the space  $M_{cb}A(G)$ is the dual of the normed space 
$(L^1(G), \|\cdot\|_{Q(G)})$ where the norm $\|\cdot\|_{Q(G)}$ is given by 
$$\|f\|_{Q(G)}=\sup\left\{\left|\int_Gf(s)\phi(s)ds\right|: \phi\in M_{cb}A(G),\|\phi\|_{cb}\leq 1 \right\},
\quad f\in L^1(G)\, .$$

We shall use the following Theorem, combining \cite[Corollary 6.3 (iv)]{spronk} and
 \cite[2.26, Corollaire 3 and 3.25, Proposition]{eym}:   
\begin{theorem}[Eymard, Spronk]
\label{quo}
 Let $G$ be a locally compact group and $H$ a closed, normal subgroup of $G$. 
 Let  $\pi: G\rightarrow G/H$ be the  quotient map. The map 
$$j_{\pi}: M_{cb}(A(G/H))\rightarrow M_{cb}(A(G)): u\mapsto u\circ\pi$$ 
 is  a complete  isometry. 
 Moreover,  $j_{\pi}(B(G/H)) \subseteq B(G)$; if $H$ is compact, then $j_{\pi}(A(G/H)) \subseteq A(G)$ .
\end{theorem}

\begin{proposition}
\label{m2} Let  $G $ be  a discrete  infinite group. Then   
 $$\sup\{\|\chi_F\|_{cb}: F\subseteq G  \text{ finite}\}=+\infty.$$  \label{m222}
\end{proposition}

\begin{proof} 
Assuming that 
$$\sup\{\|\chi_F\|_{cb}: F\subseteq G  \text{ finite}\}=M<+\infty\, ,$$
we shall prove that $\ell^{\infty}(G)\subseteq  M_{cb}(A(G))$. This means that $G$ 
is a strong Leinert set, which, by a result of  Pisier  \cite[Theorem 3.3]{pisier}, 
implies that $G$ must be finite.
 
If $v$ is an extreme point of the positive part $\Omega$ of the unit ball of $\ell^\infty (G)$, 
then $v(t)\in \{0, 1\}$ for all $t\in G.$ Thus for any finite  $F\subseteq G $, the function $\chi_Fv$ 
takes values in $\{0,1\}$ and so 
$\chi_Fv=\chi_{F'}$ for some finite subset $F'$ of $G$. Thus  
$$\|\chi_F v\|_{cb}\leq M$$ by our assumption.

Now fix an arbitrary $u\in \Omega$ and  a finite subset $ F\subseteq G$.  
By the Krein-Milman theorem, $u$ is a weak-* limit  of a net $(u_i)$ of convex combinations  of 
extreme points of $\Omega$. By the previous paragraph, each $u_i$ will satisfy 
$\|\chi_F u_i\|_{cb}\le M$.

Since $M_{cb}A(G)$ is the dual of  $(\ell^1(G), \|\cdot\|_{Q(G)})$,  given 
$\varepsilon >0$ there exists $f\in \ell^1(G)$ with  $\|f\|_{Q(G)}\leq 1$ such that
$$\|\chi_F u\|_{cb}-\varepsilon < \left|\sum_{t\in G} (\chi_F uf)(t)\right|\, .$$ Now
$$\sum_{t\in G} (\chi_F uf)(t)=\lim_i \sum_t (\chi_F u_if)(t),$$
since $u$ is the weak-* limit  of the  net $(u_i)$
and so
$$\sum_{t\in G} (\chi_F uf)(t)=\lim_i \sum_t (\chi_F u_if)(t)=\lim_i \langle f, \chi_F u_i\rangle,$$
where $\langle\cdot ,\cdot \rangle$ denotes the duality 
between $\ell^1(G)$ and $M_{cb}A(G)$.  But 
$$|\langle f, \chi_F u_i\rangle|\le\|f\|_{Q(G)} \|\chi_F u_i\|_{cb}\leq M$$ for each $i$, 
and therefore 
\begin{align*}
\|\chi_F u\|_{cb}-\varepsilon < \left|\sum_{t\in G} (\chi_F uf)(t)\right| &=
\lim_i |\langle f, \chi_F u_i\rangle| \leq M.
\end{align*}
Thus, for all nonnegative $u$ in the unit ball of $\ell^\infty (G)$ we have 
$$ \sup_{u\in \Omega} \|\chi_F u\|_{cb}\leq M$$
for every finite subset $F\subseteq G$. In particular, if $u\in c_{oo}(G)$ is nonnegative then 
$\frac{u}{\|u\|_\infty}\in \Omega$  and thus
 $$ \| u\|_{cb}\leq M\|u\|_\infty.$$ 
  Therefore  for all $u\in c_{oo}(G)$, we have that 
\begin{align*}
\|u\|_{cb} & \leq \|(\mathop{\rm Re}u)^+\|_{cb}+\|(\mathop{\rm Re}u)^-\|_{cb}
+\|(\mathop{\rm Im}u)^+\|_{cb}+\|(\mathop{\rm Im}u)^-\|_{cb}\\ 
& \leq  M(\|(\mathop{\rm Re}u)^+\|_{\infty}+\|(\mathop{\rm Re}u)^-\|_{\infty}
+\|(\mathop{\rm Im}u)^+\|_{\infty}+\|(\mathop{\rm Im}u)^-\|_{\infty}) \\
&\le 4M \|u\|_{\infty}.
\end{align*} 
 We conclude that the norms $\|\cdot \|_{cb}$ and $\|\cdot \|_{\infty }$ are equivalent on 
$c_{00}(G).$ Thus the identity map 
${\rm id}: (c_{00}(G),\|\cdot \|_{\infty})\rightarrow (M_{cb}(A(G)), \|\cdot \|_{cb}) $ is continuous. 
Since $M_{cb}(A(G))$ is a dual Banach space we can consider the  unique 
weak-* continuous extension of ${\rm id}$ to its double dual $\ell^\infty(G)$ 
(see e.g. \cite[Lemma A.2.2.]{bm})
which we denote by $T:$  
$$
T:(\ell^\infty(G),\|\cdot \|_{\infty})=(c_{00}(G),\|\cdot \|_{\infty})^{**}\to (M_{cb}(A(G)), \|\cdot \|_{cb}). $$ 
We claim that $T$ is the identity. If  $u\in\ell^\infty(G)$; we will show that $u=Tu$. 
Indeed, if $(u_i)$ is a net in $c_{00}(G)$
such that $u=\lim_i u_i$ in the weak* topology $\sigma(\ell^\infty(G),\ell^1(G)),$ 
then $(Tu_i)$ converges to $Tu$ is the weak-* topology of $M_{cb}(A(G))$, and hence pointwise 
(since $(Tu)(t)=\langle Tu, \delta_t\rangle$ for $t\in G$).  Thus, for all $t\in G$, 
 $$(Tu)(t)=\lim_i (Tu_i)(t )=\lim_i u_i(t)=u(t)$$
since $u=\lim_i u_i$ in the weak* topology $\sigma(\ell^\infty(G),\ell^1(G))$ and hence pointwise. 
This   proves our claim. 

 We have shown that $\ell^{\infty}(G)\subseteq  M_{cb}(A(G))$ and thus $G$ must be
 finite, as observed above.
\end{proof}

Since $\nor{u}_{cb}\le \nor{u}_{B(G)}$ when $u\in B(G)$ \cite[Corollary 1.8]{decahaa}, 
we obtain the following
\begin{corollary}\label{76}
 If $G $ is a discrete  infinite group then   
 $$\sup_F\{\|\chi_F\|_{A(G )}: F\subseteq G  \text{ finite}\}=+\infty.$$ 
\end{corollary}

\noindent \emph{Note } We thank the referee for providing 
the following alternative argument for Corollary \ref{76}:

Using the analogous arguments as in the proof of Proposition \ref{m2}, we can show that 
if $\sup_F\{\|\chi_F\|_{A(G )}: F\subseteq G  \text{ finite}\}<+\infty$ then
$\ell^\infty(G) =B(G)$. 

It follows from this equality that $G$ must be finite. 
Indeed if $\ell^\infty(G) =B(G)$, then $\ell^1(G) =C^*(G)$ with equivalent
norms. However, this would imply that $\ell^1(G) $ is Arens regular, which by \cite{young} shows that $G$ is finite.

\begin{theorem}\label{222} Let $G$ be an infinite totally disconnected group. Then 
$$\sup \{\|\chi_F\|_{cb}: F\subseteq G,  \chi_F \in A(G)\}=+\infty.$$ 
\end{theorem}
\begin{proof}  
From the theorem of van Dantzig \cite{dantz}, \cite[ Theorem II.7.7]{hr} there exists a 
compact open subgroup $H\subseteq G.$ 

If $H$ is finite then $\{e\}$ 
is an open subgroup of $G$ and thus $G$ is discrete. 
In this case the conclusion follows from Proposition \ref{m2}.
 
If $H$ is infinite, by Theorem 2.6 in \cite{lm} there exists a closed normal subgroup $N$ of $H$ 
such that the  quotient $H/N$  is homeomorphic to  a countably infinite product of finite  groups. 
We write $K=H/N.$ Clearly $K$ is compact and separable. We also denote by $K_d$ the group $K$ 
with the discrete topology. The inclusion $$\iota: K_d\rightarrow K$$ 
is a continuous homomorphism; thus it induces a contractive 
homomorphism 
$$\rho: A(K)\rightarrow B(K_d): u\rightarrow u\circ\iota.$$
Let  $\epsilon >0$. By Proposition \ref{m2} there exists a finite $F\subseteq K$  such that $\chi_F \in A(K_d)$
and 
$$\|\chi_F\|_{A(K_d)}>\epsilon \, .$$ 
Since $K$ is a totally disconnected and separable group, there exists a decreasing {\em sequence} 
of compact-open subgroups such that 
$$\cap_{n=1}^\infty K_n=\{e\}\quad\text{and hence }\; \cap_{n=1}^\infty FK_n=F.$$ 
Now $FK_n$ is a finite disjoint union of sets of the form   $x_iK_n$ where  $x_i\in F$ 
and since each $K_n$ is a compact open subgroup, 
$\chi_{x_iK_n}$ is in $A(K)$ and has norm 1.  
Indeed,  the constant function  1 on the compact group $K_n$ 
  belongs to $A(K_n)$ and has norm 1. It follows from \cite[Proposition 3.21 (1)]{eym} that $\chi_{K_n}$ belongs to $A(K)$ and has norm 1 and hence the same holds for  its translate $\chi_{x_iK_n}$.
 Thus 
$\chi_{FK_n}$ is in $A(K)$ and the sequence 
$(\|\chi_{FK_n}\|_{A(K)})_n$ is bounded
by the cardinality of $F$. 
Since $\rho$ is bounded, the sequence $(\|\chi_{FK_n}\|_{B(K_d)})_n$ is also bounded.
For all $f\in \ell^1(K_d)$, since  $(\chi_{FK_n})_n$ converges pointwise to $\chi_{F}$, 
by dominated convergence we have 
 $$\lim_n\sum_{t\in K_d}f(t)\chi_{FK_n}(t)=\sum_{t\in K_d}f(t)\chi_{F}(t)\, .$$
Since  $\ell^1(K_d)$ is dense in the predual $C^*(K_d)$ of $B(K_d)$, we  obtain 
$$\text{w*-}\lim_n\chi_{FK_n}=\chi_{F}$$ in the weak* topology of $B(K_d).$ 
 Therefore $\sup_n \|\chi_{FK_n}\|_{B(K_d)}>\epsilon$ and hence   
 $$\sup_n\|\chi_{FK_n}\|_{A(K)}>\epsilon $$ 
 which implies that there exists $\chi_{FK_n} \in A(K)$ such that 
 $$\|\chi_{FK_n}\|_{A(K)}>\epsilon \, .$$
This shows that $$\sup \{\|\chi_V\|_{A(H/N)}: V\subseteq H/N,  \chi_V \in A(H/N)\}=+\infty$$
and since $H/N$ is compact, it follows that
$$\sup \{\|\chi_V\|_{cb}: V\subseteq H/N,  \chi_V \in A(H/N)\}=+\infty.$$ 
Let $\pi: H\rightarrow H/N$ be the  quotient map.
It follows from Theorem \ref{quo} that if  $ F\subseteq H/N  $ is  such that 
$ \chi_F \in A(H/N)$, then $\chi_F\circ \pi=\chi_{\pi^{-1}(F)} \in A(H)$ and 
$$\|\chi_F\|_{CB(A(H/N))}=\|\chi_F\circ\pi\|_{CB(A(H))}=\|\chi_{\pi^{-1}(F)}\|_{CB(A(H))}.$$ 
We conclude that 
$$\sup \{\|\chi_{F}\|_{cb}: F\subseteq H,  \chi_F \in A(H)\}=+\infty.$$ 
Let  $F\subseteq H$ be such that $ \chi_{F} \in A(H)$. There  exists $v\in A(G)$  such that 
$v|_{H}=\chi_{F}$ \cite[3.21(1)]{ eym}. By \cite[3.21(2)]{ eym}, $\chi_{H}\in A(G)$ and hence 
$\chi_{F}=v\chi_{H}\in A(G)$.

Since by  \cite[Corollary 6.3(iii)] {spronk} the map 
$$M_{cb}A(G)\rightarrow M_{cb}A(H): u\rightarrow u|_{H}$$
is completely contractive, we obtain
$$\|\chi_{F}\|_{CB(A(G))}\geq \|\chi_{F}|_{H}\|_{CB(A(H))}=\|\chi_{F}\|_{CB(A(H))}\, .$$

We conclude that 
$$\sup \{\|\chi_{F}\|_{cb}: F\subseteq G,  \chi_F \in A(G)\}=+\infty.$$ 
\end{proof}
Note the crucial use of \cite{lm} in obtaining a {\em countable} family $(K_n)$ of compact 
open subgroups with  $\cap_{n=1}^\infty K_n=\{e\}$. 

\medskip

 The proof of the above theorem is not constructive. Below we provide a
different proof for the case where G is an infinite direct product of finite groups.
Ilie and Spronk \cite[Theorem 2.1]{IS} proved that if 
$\chi_F$ is an idempotent in  $B(G)$, then $\|\chi_F\|_{B(G)}=1$ 
if and only if $F$ is a coset of an open subgroup of $G$.
Forrest and Runde in \cite{fr} and Stan in \cite{s} proved that if the cb-norm of an idempotent 
$\chi_F \in B(G)$ satisfies $\|\chi_F\|_{cb} <\frac{2}{\sqrt{3}}$ 
then  $F$ is a coset of an open subgroup of $G$. 
 The `gap' 
$[1,\frac{2}{\sqrt{3}})$ was improved by Mudge and Pham in \cite{mp} to $[1,\frac{1+\sqrt{2}}{2})$.
\begin{proposition}\label{nonab}
Let  $G$ be an infinite direct product of  finite groups. Then 
 $$\sup \{\|\chi_F\|_{cb}: F\subseteq G \;,  \;\chi_F \in A(G)\}=+\infty.$$ 
\end{proposition}
\begin{proof} Since $G$ is compact, $A(G)=B(G)$ and    $\|u\|_{cb}=\|u\|_{A(G)}$  for all $u \in A(G)$ 
\cite[Corollary 5.4.11]{kl} and hence it is 
sufficient to prove the proposition for $\|\cdot\|_{A(G)}$. 

Let $G_0$ be a finite   group with $|G_0|\geq 3$. Then there exists $A\subseteq G_0$ such that 
$A$ is not a coset of a subgroup of $G_0$ (for example, take $A$ such that $|A|$ 
does not divide $|G_0|$). Then, since $A\in \Omega_0(G_0)$, it follows from \cite{mp} that 
$\|\chi_A\|_{A(G_0)}\geq \frac{1+\sqrt{2}}{2}$. 
 
Now consider finite  groups $G_1, G_2,\dots, G_n$ with $|G_i|\geq 3$ for all $i=1, 2, \dots, n$
and set $G=\prod_{i=1}^{n} G_i$. Choose $A_i \in G_i$ with 
$\|\chi_{A_i}\|_{A(G_i)}\geq \frac{1+\sqrt{2}}{2}$
and set $A={A_1}\times {A_2}\times\dots\times {A_n}$. 
Since $A(G)$ is isometrically  isomorphic to the operator space projective tensor product 
${A(G_1)}\hat\otimes{A(G_2)}\hat\otimes\dots\hat\otimes {A(G_n)}$ 
\cite[Lemma 4.1.2]{kl} 
we obtain 
$$\|\chi_A\|_{A(G)}=\|\chi_{A_1}\otimes\chi_{A_2}\otimes... \chi_{A_n}\|_{{A(G_1)}
\hat\otimes{A(G_2)}\hat\otimes\dots\hat\otimes {A(G_n)}}
\ge\left(\frac{1+\sqrt{2}}{2}\right)^n.$$
 
Now let $G=\prod_{i=1}^{\infty} G_i$ be an infinite product of finite groups $G_i$. 
Without loss of generality we may assume that $|G_i|\geq 3$ for all $i \in \mathbb N$ 
(lumping together some of the $G_i$'s if necessary).
Set $H_n=\prod_{i=n+1}^{\infty}G_i$ and let $\pi_n$ be the  quotient map $G\rightarrow G/H_n$.
It follows from Theorem \ref{quo}
that the map   $A(G/H_n)\rightarrow A(G): u\mapsto u\circ \pi_n$ is isometric.
Since $G/H_n\simeq\prod_{i=1}^{n} G_i$, we can choose $A \subseteq  G/H_n$ such that 
$\|\chi_A\|_{ {A(G/H_n)}}\geq (\frac{1+\sqrt{2}}{2})^n.$ Setting $F=\pi_n^{-1}(A)$, we obtain 
$\|\chi_F\|_{cb}=\|\chi_F\|_{A(G)}\geq (\frac{1+\sqrt{2}}{2})^n$ and the conclusion follows.
\end{proof}

\begin{corollary}\label{77} Let $G$ be a locally compact group and $G_0$ be the connected component of $e\in G.$ 
If the quotient $G/G_0$ is infinite then 
$$\sup \{\|\chi_F\|_{cb}: F\subseteq G, \chi_F \in B(G)\}=+\infty.$$ 
\end{corollary}
\begin{proof}
Since $G/G_0$ is infinite and  totally disconnected,  it follows from  Theorem \ref{222} that
$$\sup \{\|\chi_F\|_{cb}: F\subseteq  G/G_0, \chi_F \in A(G/G_0)\}=+\infty.$$ 

Let $\pi: H\rightarrow G/G_0$ be the quotient  map.  Since $\chi_{\pi^{-1}(F)}=\chi_F\circ\pi$, 
it follows from Theorem \ref{quo} that
$$\sup \{\|\chi_{F}\|_{cb}: F\subseteq  G, \chi_F \in B(G)\}=+\infty.$$ 
\end{proof}

\begin{remark}\label{LN3}  
Let $G$ be a locally compact group and $N$  a closed normal subgroup of $G$. It follows from Theorem \ref{quo}  that 
if
$$\sup \{\|\chi_{F}\|_{cb}: F\subseteq  G/N, \chi_F \in B(G/N)\}=+\infty$$ 
then
$$\sup \{\|\chi_{F}\|_{cb}: F\subseteq  G, \chi_F \in B(G)\}=+\infty.$$ 
\end{remark}

\section{groups with   idempotents with large norms}

Let $G$ be a locally compact group and $G_0$ the connected component of the identity of $G$.
In this section we show that $B(G)$ contains idempotents of arbitrarily large norms 
if and only if $G/G_0$ is infinite.
We also prove a related result for $A(G)$.
It follows from \cite{ho} that an  idempotent is in $B(G)$ if and only if 
it is of the form $\chi_F$ with $F$ in the open coset ring of $G$.

Let $H$ be an open subgroup  of $G$. 
Then  $H \cap G_0$ is open and closed in $G_0$, and hence equal to $G_0$; thus     
$H\supseteq G_0$.
Since  $G_0$ is contained in every open subgroup of $G$, we have that 
that if  $E$ is  an  left coset of an open subgroup in $G$, then  $E=EG_0$.
It is easy to see that 
if $E$  is a  left coset of an open subgroup $G$, we also have $E^c=E^cG_0$ 
(here $E^c$ is the complement of $E$).

\begin{lemma}\label{cr}
 Let $X$ be in the open coset ring  $ \Omega_0(G)$. Then $X=XG_0$. 
 \end{lemma}

\begin{proof}
We show that if  $X=XG_0$  and $Y=YG_0$, then  $X\cap Y=(X\cap Y)G_0$. 
Indeed, let $z \in X\cap Y$  and $g \in G_0$. Then $z=xg'=yg''$ for some $x \in X$, 
$y \in Y$   and $g', g'' \in G_0$.
Then  $zg=xg'g=yg''g$ and  since $xg'g \in X$ and $ yg''g \in Y$ we obtain   $zg \in X\cap Y$.
In view of the above remark on complements,  the assertion follows.
\end{proof}

\begin{corollary}\label{fin}
Let $\phi$ be the map defined on $\Omega_0(G)$ by    $ X\mapsto q(X)$ 
(where $q:G\to G/G_0$ is the quotient map). Then  $\phi$ is a ring isomorphism from 
$\Omega_0(G)$ onto $\Omega_0(G/G_0)$.
 \end{corollary}
\begin{proof}
  Clearly $\phi(X\cap Y)=\phi(X)\cap\phi(Y)$ and $\phi(X^c)=\phi(X)^c$.
Let $aK$ be a coset of an open subgroup in $G/G_0$. 
Then $\phi^{-1}(aK)$ is a coset of an open subgroup in $G$ and 
$\phi(\phi^{-1}(aK))=aK$. Finally, $XG_0=YG_0$ for $X, Y \in \Omega_0(G)$ implies $X=Y$, hence $\phi$ is injective.
\end{proof}

\begin{theorem}\label{equiv}
 Let $G$ be a locally compact group and $G_0$ be the connected component of $e\in G.$ 
 The following are equivalent
 \begin{enumerate}
  \item    The quotient $G/G_0$ is infinite and $G_0$ is compact
    
  \item  $\;\sup \{\|\chi_F\|_{cb}: \chi_F\in A(G)\}=+\infty.$
  
  \item  $\;\sup \{\|\chi_F\|_{A(G)}: \chi_F \in A(G)\}=+\infty.$
 \end{enumerate}
\end{theorem}

\begin{proof} That (1) implies (2) follows from Theorems \ref{quo} and \ref{222}. 
   
That  (2) implies (3) follows since $\|\chi_F\|_{cb}\le\|\chi_F\|_{A(G)}$.

We show that (3) implies (1): If $G_0$ is not compact, it follows from Lemma \ref{cr} 
that  there are no idempotents 
in $A(G)$. If $G/G_0$ is finite,  the open coset ring is finite from Corollary \ref{fin}, and hence 
the set of idempotents is finite.
\end{proof}
  
 The proof of the following theorem is similar.
\begin{theorem}\label{equiv2}
 Let $G$ be a locally compact group and $G_0$ be the connected component of $e\in G.$ 
 The following are equivalent
   \begin{enumerate}
  \item    The quotient $G/G_0$ is infinite.
      \item  $\;\sup \{\|\chi_F\|_{cb}: \chi_F\in B(G)\}=+\infty.$
  \item  $\;\sup \{\|\chi_F\|_{B(G)}: \chi_F \in B(G)\}=+\infty.$
 \end{enumerate}
\end{theorem}

\section{norms of  homomorphisms}
In this section we show that if $G$ is a locally compact group 
with connected component of the identity $G_0$ such that $G/G_0$ is infinite, 
and $H$ is a locally compact group, then 
there exist homomorphisms of arbitrarily large norm from $A(H)$ into $B(G)$. We also prove that if 
there  exists an amenable group $H$ such that  homomorphisms of arbitrarily large 
norm from $A(H)$ into $B(G)$ exist, then  $G/G_0$ is infinite.

\begin{proposition} \label{511}
Let $G, H$  be locally compact groups and  $F\in\Omega_0(G)$. For  $u\in A(H)$ we define 
$$\rho_F(u)(t)=\left\{\begin{array}{ll} u(e) , & t\in F\\  
0, & t\notin F\end{array}\right.$$
Then  the map $\rho_F$ is  a completely bounded homomorphism $A(H)\rightarrow B(G)$ and  
$$\|\rho_F\|_{cb}=\|\rho_F\|=\|\chi_F\|_{B(G)}.$$
\end{proposition}
\begin{proof} 
It follows from   \cite[Theorem 3.1]{IS} that the map $\rho_F$ is  a completely bounded homomorphism. 
Choose $u\in A(H)$ such that $u(e)=1$ and  $\|u\|_{A(H)}\leq 1 $. Then 
$$ \|\rho_F\|\geq  \|\rho_F(u)\|_{B(G)}= \|u(e)\chi_F\|_{B(G)}= \|\chi_F\|_{B(G)}.$$ 
We also have, for  $u\in A(H)$, 
$$\|\rho_F(u)\|_{B(G)}=\|u(e)\chi_F\|_{B(G)}=|u(e)|\|\chi_F\|_{B(G)}\leq \|u\|_{A(H)}\|\chi_F\|_{B(G)}\, ,$$ 
and hence
$\|\rho_F\|=\|\chi_F\|_{B(G)}.$ 

  Since the image of $\rho_F$ is  one-dimensional,  
it follows that   $$\|\rho_F\|_{cb}=\|\rho_F\|=\|\chi_F\|_{B(G)}.$$
\end{proof}
Applying Theorem \ref{equiv2} to Proposition \ref{511}, we obtain the following
\begin{corollary}
\label{8}
Let $G, H$  be  locally compact groups and assume that 
$$\sup \{\|\chi_F\|_{cb}: F\in\Omega_0(G)\}=+\infty \, .$$  Then 
$$\sup\{\|\rho: A(H)\rightarrow B(G)\|: \rho  \text{ is a cb homomorphism}\}=+\infty \, .$$
\end{corollary}

\begin{theorem}\label{LN4}  Let $G$ be a locally compact group
and $G_0$ be the connected component of $e\in G.$ The following are equivalent: 
$$ (i) \quad \sup\{\|\rho: A(H)\rightarrow B(G)\| :\rho \text{ is a cb homomorphism}\}=+\infty.$$
for every locally compact group $H$.

(ii) There exists an amenable locally compact group $H$ such that 
$$\sup\{\|\rho: A(H)\rightarrow B(G)\| :\rho \text{ is a cb homomorphism}\}=+\infty.$$

(iii) The group $G/G_0$ is infinite.

(iv) $\sup \{\|\chi_F\|_{B(G)}: F\in \Omega_0(G)\}=+\infty.$ 

\end{theorem}
\begin{proof}
Clearly $(i)$  implies $ (ii)$. The equivalence $(iii)\Leftrightarrow (iv)$  
follows from Theorem \ref{equiv2}. 
Also the implication $(iv)\Rightarrow (i)$ follows from Corollary \ref{8}. 

It remains to show that $(ii)$ implies $(iii)$. 
Suppose not, i.e. that $|G/G_0|<+\infty.$ 
By Corollary \ref{fin}, there exists $m \in \mathbb N$ such that $|\Omega_0(G)|\leq m.$
Let $ \rho: A(H)\rightarrow B(G)$  be a completely bounded homomorphism.  
By  \cite[Theorem 3.7]{IS}, $\rho$ is of the 
form \ref{star2} for some $Y \in \Omega_0(G)$ and a piecewise affine map $\alpha: Y\rightarrow H$. 
By Proposition 3.1 in \cite{IS} we have that 
$$\|\rho \|_{cb}\leq m\cdot\!\! \sum_{F\in \Omega_0(G)}\|\chi_F\|_{B(G)}
\leq m^2 \max\{\|\chi_F\|_{B(G)}: F \in \Omega_0(G)\}$$
which is a   contradiction.
\end{proof}

\noindent
{\bf Acknowledgment } Warm thanks are due to  the  referee for his useful comments and in particular for indicating a gap in the original proof of  Theorem \ref{222}.


\begin{thebibliography}{10}

\bibitem{bm}
{\sc Blecher, D.~P. and Le~Merdy, C.}
\newblock {\em Operator algebras and their modules---an operator space
  approach}, vol.~30 of {\em London Mathematical Society Monographs. New
  Series}.
\newblock The Clarendon Press, Oxford University Press, Oxford, 2004.
\newblock Oxford Science Publications.
\newblock  \href{https://doi.org/10.1093/acprof:oso/9780198526599.001.0001}{doi:10.1093/acprof:oso/9780198526599.001.0001} 

\bibitem{co1} 
{\sc Cohen, P.~J.}
\newblock On a conjecture of Littlewood and idempotent measures. 
\newblock{\em Amer. J. Math. 82 \/}(1960), 191--212. 
\newblock \href{https://doi.org/10.2307/2372731}{doi:10.2307/2372731}

\bibitem{co2}
{\sc Cohen, P.~J.}
\newblock On homomorphisms of group algebras.
\newblock {\em Amer. J. Math. 82\/} (1960), 213--226.
\newblock \href{https://doi.org/10.2307/2372732}{doi:10.2307/2372732}

\bibitem{dantz} 
{\sc Van Dantzig, D.}
 \newblock  {Zur topologischen Algebra. III. Brouwersche und Cantorsche   Gruppen.}
\newblock {\em Compositio Math. 3 (1936), 408--426.}
\newblock \href{http://www.numdam.org/item/CM_1936__3__408_0/}{numdam:CM-1936--3--408-0}

\bibitem{daws}
{\sc Daws, M.}
\newblock Completely bounded homomorphisms of the fourier algebra revisited,  2020.
\newblock \href{https://arxiv.org/pdf/2011.03962.pdf}{arXiv:2011.03962 [math.FA]}

  \bibitem{decahaa} 
 {\sc {de Canni\`ere}, J.  and Haagerup, U.}
\newblock {Multipliers of the Fourier algebras of some simple Lie groups and
  their discrete subgroups.}
\newblock {\em Amer. J. Math. 107 \/} (1985), 455--500.
\newblock \href{https://doi.org/10.2307/2374423}{doi:10.2307/2374423}

\bibitem{eym}
{\sc Eymard, P.}
\newblock L'alg\`ebre de {F}ourier d'un groupe localement compact.
\newblock {\em Bull. Soc. Math. France 92\/} (1964), 181--236.
 \newblock \href{https://doi.org/10.24033/bsmf.1607}{doi:10.24033/bsmf.1607}

\bibitem{fr} 
{\sc Forrest, B. E.  and Runde V.}  
\newblock Norm one idempotent cb-multipliers with applications to the 
Fourier algebra in the cb-multiplier norm.
\newblock {\em Canad. Math. Bull. 54 \/} (2011), no. 4, 654--662. 
\newblock \href{https://doi.org/10.4153/CMB-2011-098-0}{doi:10.4153/CMB-2011-098-0}

\bibitem{hr}
 {\sc Hewitt, E.  and  Ross, K. A.} 
 \newblock {\em Abstract harmonic analysis. Vol. I. 
 Structure of topological groups, integration theory, group representations}.
\newblock Second edition.
 \newblock{\em Grundlehren der Mathematischen Wissenschaften} , 
 115. Springer-Verlag, Berlin-New York, 1979. ix+519 pp. 
\newblock \href{https://doi.org/10.1007/978-1-4419-8638-2}{doi:10.1007/978-1-4419-8638-2}

\bibitem{ho} 
{\sc Host, B.}   
\newblock Le th\'{e}or\`{e}me des idempotents dans $B(G)$. 
\newblock{\em  Bull. Soc. Math. France 114} (1986), no. 2, 215–223. 
\newblock \href{https://doi.org/10.24033/bsmf.2055}{doi:10.24033/bsmf.2055}

\bibitem{I}
{\sc Ilie, M.}
\newblock On {F}ourier algebra homomorphisms.
\newblock {\em J. Funct. Anal. 213}, 1 (2004), 88--110.
\newblock \href{https://doi.org/10.1016/j.jfa.2004.04.013}{doi:10.1016/j.jfa.2004.04.013}

\bibitem{IS}
{\sc Ilie, M. and Spronk, N.}
\newblock Completely bounded homomorphisms of the {F}ourier algebras.
\newblock {\em J. Funct. Anal. 225}, 2 (2005), 480--499.
\newblock \href{https://doi.org/10.1016/j.jfa.2004.11.011}{doi:10.1016/j.jfa.2004.11.011}

\bibitem{kl}
{\sc Kaniuth, E. and Lau, A. T.-M.}
\newblock {\em Fourier and {F}ourier-{S}tieltjes algebras on locally compact
  groups}, vol.~231 of {\em Mathematical Surveys and Monographs}.
\newblock American Mathematical Society, Providence, RI, 2018.
\newblock \href{https://doi.org/10.1090/surv/231}{doi:10.1090/surv/231}

\bibitem{lm} 
{\sc Leiderman, A.G., Morris, S.A.  and Tkachenko, M.G. }
\newblock The separable quotient problem for topological groups.
\newblock {\em Israel J. of Mathematics}  234 (2019), 331--369.
\newblock \href{https://doi.org/10.1007/s11856-019-1931-1}{doi:10.1007/s11856-019-1931-1}

\bibitem{mp} 
{\sc Mudge, J.   and  Pham, H. L.}  
\newblock Idempotents with small norms. 
\newblock{\em J. Funct. Anal.}
270 (2016), no. 12, 4597–4603. 
\newblock \href{https://doi.org/10.1016/j.jfa.2016.02.011}{doi:10.1016/j.jfa.2016.02.011}

 \bibitem{pisier} 
 {\sc Pisier, G.}
\newblock Multipliers and Lacunary sets in non-amenable groups.
\newblock {\em American Journal of Mathematics}, 117(2) (1995) 337--376.
\newblock \href{https://doi.org/10.2307/2374918}{doi:10.2307/2374918}

\bibitem{s}  
{\sc Stan, A-M. P.} 
\newblock On idempotents of completely bounded multipliers of the Fourier algebra A(G).
\newblock{\em  Indiana Univ. Math. J. 58 \/} (2009), no. 2, 523--535.
\newblock \href{https://doi.org/10.1512/iumj.2009.58.3452}{doi:10.1512/iumj.2009.58.3452}

\bibitem{spronk}
{\sc Spronk, N.}
\newblock  Measurable Shur multipliers and completely bounded multipliers of the Fourier algebras.  
\newblock {\em Proc. London Math. Soc.}, 89 (3): 161-192, 2004. 
\newblock \href{https://doi.org/10.1112/S0024611504014650}{doi:10.1112/S0024611504014650}
 
\bibitem{young}
 {\sc Young, N. J.}, \newblock The irregularity of multiplication in group algebras.
 \newblock {\em Quart. J.
Math. Oxford} Ser. (2) 24 (1973), 59–62.
\newblock \href {https://doi.org/10.1093/qmath/24.1.59}{doi:10.1093/qmath/24.1.59}


\end{thebibliography}
\end{document}